\documentclass[11pt]{amsart}
\theoremstyle{definition}
\usepackage{indentfirst, amssymb,amsmath,amsthm}
\usepackage{csquotes}
\usepackage{hyperref}
\usepackage{mathrsfs}
\newtheorem{theo}{Theorem}[section]
\newtheorem{lem}{Lemma}[section]

\newtheorem{defi}{Definition}[section]

\newtheorem{ques}{Question}[section]
\newcommand{\ol}{\overline}
\newcommand{\be}{\begin{equation}}
\newcommand{\ee}{\end{equation}}
\newcommand{\beas}{\begin{eqnarray*}}
\newcommand{\eeas}{\end{eqnarray*}}
\newcommand{\bea}{\begin{eqnarray}}
\newcommand{\eea}{\end{eqnarray}}

\numberwithin{equation}{section}
\begin{document}
\title[Uniqueness of meromorphic functions]{Uniqueness of meromorphic functions that share Two Sets}
\date{}
\author[B. Chakraborty, A. K. Pal and S. Saha]{Amit Kumar Pal$^{1}$, Bikash Chakraborty$^{2}$ and Sudip Saha$^{3}$}
\date{}
\address{$^{1}$Department of Mathematics, University of Kalyani, Kalyani, West Bengal 741235, India}
\email{mail4amitpal@gmail.com}

\address{$^{2}$Department of Mathematics, Ramakrishna Mission Vivekananda Centenary College, Rahara,
West Bengal 700 118, India.}
\email{bikashchakraborty.math@yahoo.com, bikash@rkmvccrahara.org}

\address{$^{3}$Department of Mathematics, Ramakrishna Mission Vivekananda Centenary College, Rahara,
West Bengal 700 118, India.}
\email{sudipsaha814@gmail.com, sudip.math@rkmvccrahara.ac.in}
\maketitle
\let\thefootnote\relax
\maketitle
\begin{abstract}
In this note, we introduce a new kind of pair of finite range sets in $\mathbb{C}$ for meromorphic functions corresponding to their uniqueness, i.e.,  how two
meromorphic functions are uniquely determined by their two finite shared sets.
\end{abstract}
\section{Introduction}
We use $M(\mathbb{C})$ to denote the field of all meromorphic functions in $\mathbb{C}$. Let $S\subset\mathbb{C}\cup\{\infty\}$ be a non-empty set with distinct elements and $f\in M(\mathbb{C})$. We set $$E_{f}(S)=\bigcup\limits_{a\in S}\{z~:~f(z)-a=0\},$$
where a zero of $f-a$ with multiplicity $m$ counts $m$ times in $E_{f}(S)$. Let $\ol{E}_{f}(S)$ denote the collection of distinct elements in $E_{f}(S)$.\par
Let $g\in M(\mathbb{C})$. We say that two functions $f$ and $g$ share the set $S$ CM (resp. IM) if $E_{f}(S)=E_{g}(S)$ (resp. $\overline{E}_{f}(S)=\overline{E}_{g}(S)$).\par
Let $f$ and $g$ be two nonconstant entire functions and  $S_{1}=\{1\}$, $S_{2}=\{-1\}$ and $S_{3}=\{a, b\in \mathbb{C}~:~a\not=\pm1, b\not=\pm1, b\not=\frac{1}{a}, b\not=1+\frac{4}{a-1}\}$. F. Gross (\cite{4.01}) proved that  if $E_{f}(S_{j})=E_{g}(S_{j})$ (for $j=1,2,3$), then $f=g$.\par
In 1982, F. Gross and C. C. Yang (\cite{GY}) first proved that if two nonconstant entire functions $f$ and $g$ share the set $S=\{z\in \mathbb{C}:e^{z}+z=0\}$, then $f\equiv g$.\par
Moreover, this type of set was termed as a unique range set for entire functions. Later, similar definition for meromorphic functions was also introduced in the literature.
\begin{defi}(\cite{YY})
Let $S\subset \mathbb{C}\cup\{\infty\}$; $f$ and $g$ be two nonconstant meromorphic (resp. entire) functions. If $E_{f}(S)=E_{g}(S)$ implies $f\equiv g$, then $S$ is called a unique range set for meromorphic (resp. entire) functions or in brief URSM (resp. URSE).
\end{defi}
The introduction of the unique range set for meromophic function (URSM in brief) influenced many researchers to  to find finite URSM’s (e.g., see, \cite{1, 1.1, 1.2, 2, 2.1, 2.2} etc.)
\section{Functions sharing two sets}

In continuation with the notion of unique range sets, in 2003, the following question was asked by Lin and Yi in (\cite{lin}).
\begin{ques} \label{q1}
  Can one find two finite sets $S_j (j = 1, 2)$ such that any two nonconstant meromorphic functions $f$ and $g$ satisfying $E_{f}(S_{j})=E_{g}(S_{j})$ for
$j = 1, 2$ must be identical ?
\end{ques}
To answer the Question \ref{q1}, a lot of investigations has been done to give explicitly a set $S$ with $n$ elements and make $n$ as small as possible such that any two meromorphic functions $f$ and $g$ that share the value $\infty$ and the set $S$ must be equal.\par
In this direction, in 2012, B. Yi and Y. H. Li (\cite{2YL}) provided a significant result. They proved that there exists two finite sets $S_{1}$ (with $5$ elements) and $S_{2}$ (with $2$ elements) such that if any two nonconstant meromorphic functions $f$ and $g$ satisfying the condition $E(S_{j}, f) = E(S_{j}, g)$ for $j = 1, 2$, then $f\equiv g$.
\begin{theo}(\cite{2YL}) \label{th1}
Let $S_{1}=\{0,1\}$ and $S_{2}=\{z~:~\frac{(n-1)(n-2)}{2}z^{n}-n(n-2)z^{n-1}+\frac{n(n-1)}{2}z^{n-2}+1=0\}$, where $n(\geq 5)$ is an integer. If two nonconstant meromorphic functions $f$ and $g$ share $S_1$ CM and $S_2$ CM, then $f\equiv g$.
\end{theo}
Now, we note that if $P(z)=\frac{(n-1)(n-2)}{2}z^{n}-n(n-2)z^{n-1}+\frac{n(n-1)}{2}z^{n-2}+1$, then $P'(z)=\frac{n(n-1)(n-2)}{2}z^{n-3}(z-1)^{2}$. Thus $S_{1}$ is the zero set of  $P'(z)$. With this obserbvation, the following question is obvious ?\par
Suppose
\begin{equation}\label{eb1}
P(z)=a_{0}(z-\alpha_{1})(z-\alpha_{2})\ldots(z-\alpha_{n}),
\end{equation}
 where $\alpha_{i}\not=\alpha_{j}$, $1\leq i,j\leq n$; $a_{0}\not=0$. Further suppose that
\begin{equation}\label{eb2}
P'(z)=b_{0}(z-\beta_{1})^{q_1}(z-\beta_{2})^{q_2},
 \end{equation}
\begin{ques} \label{q2}
 Is it possible to take $S_{1}=\{\beta_{1},\beta_{2}\}$ (with $\beta_{1}\not= \beta_{2}$) and $S_{2}=\{\alpha_{1},\alpha_{2},\ldots,\alpha_{n}\}$  in Theorem \ref{th1} such that  any two nonconstant meromorphic functions $f$ and $g$ satisfying $E_{f}(S_{j})=E_{g}(S_{j})$ for
$j = 1, 2$ must be identical ?
\end{ques}
In (\cite{chak}), the authors dealt with the Question \ref{q2}. Let $P(z)$ be a nonconstant monic polynomial. We say $P(z)$ as a uniqueness polynomial if  $P(f)\equiv P (g)$ implies $f \equiv g$ for any two nonconstant meromorphic functions $f,~g$; while a strong uniqueness polynomial if  $P(f)\equiv cP (g)$ implies $f \equiv g$ for any two nonconstant meromorphic functions $f,~g$ and non-zero constant $c$. Moreover, we say that $P(z)$ is satisfying the Fujimoto's Hypothesis (\cite{2.1}) if
\begin{equation}\label{eb3}
  P(\beta_{l_{s}})\not=P(\beta_{l_{t}})~~~~(1\leq l_{s}< l_{t}\leq k).
\end{equation}

\begin{theo}\cite{chak}
Let $P(z)$ be a strong uniqueness polynomial of the form (\ref{eb1}) satisfying the condition (\ref{eb3}) and $P'(z)$ have no simple zeros. Further suppose that $S_{1}=\{\beta_{1},\beta_{2}\}$ and $S_{2}=\{\alpha_{1},\alpha_{2},\ldots,\alpha_{n}\}$.\par
If two nonconstant meromorphic functions $f$ and $g$ share the set $S_{1}$ IM and $S_{2}$ CM,  and $n\geq 6$, then $f\equiv g$.
\end{theo}
But, in literature, the examples of strong uniqueness polynomials are few. Thus the motivation of this paper is to present some uniqueness theorems which show how two meromorphic functions are uniquely determined by their two shared sets. \par
Let us consider the following polynomial which was introduced in (\cite{1.2}) by A. Banerjee and B. Chakraborty. \\
\begin{equation}\label{eq1n}
P(z):=Q(z)+c:=\sum\limits_{i=0}^{m} \sum\limits_{j=0}^{n} \binom{m}{i}\binom{n}{j}\frac{(-1)^{i+j}}{n+m+1-i-j}z^{n+m+1-i-j}a^{j}b^{i}+c,
\end{equation}
where $a, b$ be two complex numbers such that $b\neq0$, $a\not=b$ and
$$c\not\in \{0,-Q(a),-Q(b),-\frac{Q(a)+Q(b)}{2}\}.$$
The above mentioned polynomial is the generalization of the unique range set generating polynomial introduced in (\cite{1}, \cite{1.1}). Clearly \beas P'(z) &=& \sum\limits_{i=0}^{m} \sum\limits_{j=0}^{n} \binom{m}{i}\binom{n}{j}(-1)^{i+j}z^{m+n-i-j}a^{j}b^{i}\\
&=& \big(\sum\limits_{i=0}^{m}(-1)^{i}\binom{m}{i}z^{m-i}b^{i}\big)\big(\sum\limits_{j=0}^{n}\binom{n}{j}(-1)^{j}z^{n-j}a^{j}\big)\\
&=& (z-b)^{m}(z-a)^{n}\eeas
Thus by the assumptions on $c$, $P(z)$ has only simple zeros.\par Next, we can write
$P(z)-P(b)=(z-b)^{m+1}R_{1}(z)$ where  $R_{1}(b)\not=0$, and
 $P(z)-P(a)=(z-a)^{n+1}R_{2}(z)$ where $R_{2}(a)\not=0$.\par
If $P(a)= P(b)$, then $(z-b)^{m+1}R_{1}(z)=(z-a)^{n+1}R_{2}(z)$, i.e.,  $R_{1}(z)$ has a factor $(z-a)^{n+1}$ which implies the polynomial $P$ is of degree at least $m+1+n+1$, a contradiction. Thus $P(a)\not=P(b)$.\par
 In (\cite{1.2}), Banerjee and Chakraborty proved the following Theorem:
\begin{theo}(\cite{1.2})
Let $m+n\geq 10$ (respectively. 6), $\max\{m,n\} \geq 3$ and $\min\{m,n\}\geq 2$ and $S=\{z~:~P(z)=0\}$ where $P(z)$ is defined in (\ref{eq1n}). If, for two nonconstant meromorphic (respectively. entire) functions $f$ and $g$, $E_{f}(S)=E_{g}(S)$, then $f\equiv g$.
\end{theo}
Inspired by the above theorem, we observe the following theorem which is one of the main theorem of this paper:
\begin{theo}\label{Th1}
Let $S_1=\{a, b\}$ and $S_2=\{z \in \mathbb{C}: P(z)=0\}$, where $P(z)$ is defined in (\ref{eq1n}) with the already defined choice of $a$, $b$. If two nonconstant meromorphic functions $f$ and $g$ share $S_1$ IM and $S_2$ CM, and, either  $n\geq3$ and $m\geq 2$, or $n\geq2$ and $m\geq 3$, then $f\equiv g$.
\end{theo}
Now, we explain some definitions and notations which we need to prove the Theorem \ref{Th1}.
\begin{defi} \label{d3}\cite{YY} Let $a\in\mathbb{C}\cup\{\infty\}$ and $m\in \mathbb{N}$.
\begin{enumerate}
\item [i)] By $N(r,a;f\mid=1)$, we mean the counting function of the simple $a$-points of $f$.
\item [ii)] By $N(r,a;f\mid\leq m)$ (resp. $N(r,a;f\mid\geq m)$, we denote the counting function of those $a$-points of $f$ whose multiplicities are not greater (resp. less) than $m$ where each $a$-point is counted according to its multiplicity.
\end{enumerate}
 Similarly, $\ol N(r,a;f \mid\leq m)$ and $\ol N(r,a;f \mid\geq m)$ are the reduced counting function of $N(r,a;f\mid\leq m)$ and $N(r,a;f\mid\geq m)$ respectively.
\end{defi}
\begin{defi}\label{d5}\cite{YY} Let $f$ and $g$ be two nonconstant meromorphic functions such that $f$ and $g$ share $a$ IM. Let $z_{0}$ be an $a$-point of $f$ with multiplicity $p$, an $a$-point of $g$ with multiplicity $q$.
\begin{enumerate}
  \item [i)] By $\ol N_{L}(r,a;f)$, we mean the reduced counting function of those $a$-points of $f$ and $g$ where $p>q$.
  \item [ii)] By $N^{1)}_{E}(r,a;f)$, we mean the counting function of those $a$-points of $f$ and $g$ where $p=q=1$.
  \item [iii)] By $\ol N^{(2}_{E}(r,a;f)$, we mean the reduced counting function of those $a$-points of $f$ and $g$ where $p=q\geq 2$.
\end{enumerate}
Similarly, we can define $\ol N_{L}(r,a;g)$, $N^{1)}_{E}(r,a;g)$, $\ol N^{(2}_{E}(r,a;g)$. When $f$ and $g$ share $a$ with weight $m$ $(m\geq 1)$, then $$N^{1)}_{E}(r,a;f)=N(r,a;f\mid=1).$$
\end{defi}
\begin{defi} \label{d7}\cite{YY} Let $f$, $g$ share a value $a$ IM. We denote by $\ol N_{*}(r,a;f,g)$ the reduced counting function of those $a$-points of $f$ whose multiplicities differ from the multiplicities of the corresponding $a$-points of $g$. Clearly $$\ol N_{*}(r,a;f,g) \equiv \ol N_{*}(r,a;g,f)~~\text{and}~~\ol N_{*}(r,a;f,g)=\ol N_{L}(r,a;f)+\ol N_{L}(r,a;g).$$
\end{defi}
Now, we state some necessary lemmas to prove the main result.
\begin{lem}\label{l2.3}\cite{4.2} Let $f$ be a nonconstant meromorphic function and let \[R(f)=\frac{\sum\limits _{k=0}^{n} a_{k}f^{k}}{\sum \limits_{j=0}^{m} b_{j}f^{j}}\] be an irreducible rational function in $f$ with constant coefficients $\{a_{k}\}$ and $\{b_{j}\}$ where $a_{n}\not=0$ and $b_{m}\not=0$. Then $$T(r,R(f))=d\cdot T(r,f)+S(r,f),$$ where $d=\max\{n,m\}$.
\end{lem}
\begin{lem}(\cite{2.1}, Proposition 7.1.)\label{lem1}
Let $P(z)$ be a nonzero polynomial of degree $q\geq 5$ without multiple zeros whose derivative is given by $$P'(z)=q(z-d_{1})^{q_{1}}(z-d_{2})^{q_{2}},$$ where $d_{1}\not= d_{2}$ and $q_{1}+q_{2}=q-1$.  Assume that $P(d_{1})\not=P(d_{2})$ and there are two distinct nonconstant meromorphic functions $f$ and $g$ such $$\frac{1}{P(f)}=\frac{c_{0}}{P(g)}+c_{1}$$
for some constants $c_{0}\not=0$ and $c_{1}$. If $\min\{q_{1},q_{2}\}\geq 2$, then $c_{1}=0$.
\end{lem}

\begin{lem}\label{lem3}
Let $P(z)$ be a polynomial defined in (\ref{eq1n}). If for any two nonconstant meromorphic functions $f$ and $g$, $P(f)\equiv AP(g)$, where $A$ is any nonzero constant, and, either $n\geq 3$, $m\geq 1$, or, $m\geq 3$, $n\geq 1$, then $A=1$.
\end{lem}
\begin{proof} The idea of the proof is from (\cite{1.2}). By choice of $a, b, c$, we note that $P(a)\not=P(b)$, $P(a)+P(b)\not=0$ and $P(z)$ has no multiple zero. Also, we note that
$$P(z)-P(b)=(z-b)^{m+1}R_{1}(z),$$ where  $R_{1}(b)\not=0$, and
$$P(z)-P(a)=(z-a)^{n+1}R_{2}(z),$$ where $R_{2}(a)\not=0$. Now, we consider two cases:\\
\textbf{Case-I} Assume that $n\geq 3$ and $m\geq1$. We define
$$F_{1}=\frac{(f-a)^{n+1}R_{2}(f)}{P(a)},~~\text{and}~~G_{1}=\frac{(g-a)^{n+1}R_{2}(g)}{P(a)}.$$
Then $$F_{1}+1\equiv A(G_{1}+1).$$
If $A\not=1$ and $A\not=\frac{P(b)}{P(a)}$, then using the first fundamental theorem, second fundamental theorem and Mokhon'ko Lemma, we have
\beas &&2(m+n+1)T(r,f)+O(1)\\
&=&2T(r,F_{1})\\
&\leq& \ol N(r,\infty;F_1)+\ol N(r,0;F_1)+\ol N\left(r,\frac{P(b)}{P(a)}-1;F_1\right)+\ol N(r,A-1;F_1)+S(r,F_1)\\
&\leq&\ol N(r,\infty;f)+(m+1)T(r,f)+(n+1)T(r,f)+(m+1)T(r,g)+S(r,f)\\
&\leq& (n+2m+4)T(r,f)+S(r,f),\eeas
which is impossible as $n\geq 3$. Thus we assume that $A\not=1$ and $A=\frac{P(b)}{P(a)}$, i.e.,
$$F_{1}\equiv \frac{P(b)}{P(a)}\left(G_{1}+1-\frac{P(a)}{P(b)}\right).$$
Now, using the first fundamental theorem, second fundamental theorem and Mokhon'ko Lemma, we have
\beas &&2(m+n+1)T(r,g)+O(1)\\
&=&2T(r,G_{1})\\
&\leq& \ol N(r,\infty;G_1)+\ol N(r,0;G_1)+\ol N\left(r,\frac{P(a)}{P(b)}-1;G_1\right)+\ol N(r,\frac{P(b)}{P(a)}-1;G_1)+S(r,G_1)\\
&\leq&\ol N(r,\infty;g)+(m+1)T(r,g)+(m+1)T(r,f)+(n+1)T(r,g)+S(r,g)\\
&\leq& (n+2m+4)T(r,g)+S(r,g),\eeas
which is impossible as $n\geq 3$. Hence  $A=1$.\\
\textbf{Case-II} Assume that $m\geq 3$ and $n\geq1$. We define
$$F_{2}=\frac{(f-b)^{m+1}R_{1}(f)}{P(b)},~~\text{and}~~G_{2}=\frac{(g-b)^{m+1}R_{1}(g)}{P(b)}.$$
Then $$F_{2}+1\equiv A(G_{2}+1).$$
If $A\not=1$ and $A\not=\frac{P(a)}{P(b)}$, then using the first fundamental theorem, second fundamental theorem and Mokhon'ko Lemma, we have
\beas &&2(m+n+1)T(r,f)+O(1)\\
&=&2T(r,F_{2})\\
&\leq& \ol N(r,\infty;F_2)+\ol N(r,0;F_2)+\ol N\left(r,\frac{P(a)}{P(b)}-1;F_2\right)+\ol N(r,A-1;F_2)+S(r,F_2)\\
&\leq&\ol N(r,\infty;f)+(n+1)T(r,f)+(m+1)T(r,f)+(n+1)T(r,g)+S(r,g)\\
&\leq& (2n+m+4)T(r,f)+S(r,f),\eeas
which is impossible as $m\geq 3$. Thus we assume that $A\not=1$ and $A=\frac{P(a)}{P(b)}$, i.e.,
$$F_{2}\equiv \frac{P(a)}{P(b)}\left(G_{2}+1-\frac{P(b)}{P(a)}\right).$$
Now, using the first fundamental theorem, second fundamental theorem and Mokhon'ko Lemma, we have
\beas &&2(m+n+1)T(r,g)+O(1)\\
&=&2T(r,G_{2})\\
&\leq& \ol N(r,\infty;G_2)+\ol N(r,0;G_2)+\ol N\left(r,\frac{P(a)}{P(b)}-1;G_2\right)+\ol N(r,\frac{P(b)}{P(a)}-1;G_2)+S(r,G_2)\\
&\leq&\ol N(r,\infty;g)+(n+1)T(r,g)+(m+1)T(r,g)+(n+1)T(r,f)+S(r,g)\\
&\leq& (2n+m+4)T(r,f)+S(r,f),\eeas
which is impossible as $m\geq 3$. Hence  $A=1$.\\
\end{proof}
\begin{lem} (\cite{2.2})\label{lem2}
Let $P(z)$ be a nonzero monic polynomial  of degree $q\geq 6$ without multiple zeros  whose derivative is given by $$P'(z)=q(z-d_{1})^{q_{1}}(z-d_{2})^{q_{2}},$$ where $d_{1}\not= d_{2}$.  Assume that $P(d_{1})\not=P(d_{2})$ and $\min\{q_{1},q_{2}\}\geq 2$, $q_{1}+q_{2}\geq 5$. If for any two nonconstant meromorphic functions $f$ and $g$, $P(f)\equiv P(g)$, then $f\equiv g$.
\end{lem}

\begin{proof} [\textbf{Proof of the Theorem \ref{Th1}}]
Let us define the following functions:
 $$F=-\frac{Q(f)}{c} ~~ \text{and} ~~ G=-\frac{Q(g)}{c},$$
 and
 $$H=\left(\frac{F''}{F'}-\frac{2F'}{F-1}\right)-\left(\frac{G''}{G'}-\frac{2G'}{G-1}\right) ~~~\text{and}~~~
\psi= \left(\frac{F'}{F-1}-\frac{G'}{G-1}\right).$$
Since $f, g$ share $S_2$ CM, then $F, G$ share $1$ CM.\\
\medbreak
\textbf{Case-I} 
If possible, let $H   \not \equiv 0$. Then obviously $F-1$ and $G-1$ are not linearly dependent. Now, if possible, let   $\psi \equiv0$. Then we obtain
$$(F-1)=A(G-1),$$
where $A$ is a non-zero complex number, which is impossible as $H\not\equiv0$. Hence $$\psi \not\equiv 0.$$
Thus $m(r, \psi)=S(r,f)+S(r,g)$. Now, we see that
\begin{eqnarray*}
\psi =  \frac{(f-b)^m (f-a)^n f'}{c(F-1)}-\frac{(g - b)^m (g-a)^n g'}{c(G-1)}.
\end{eqnarray*}
Let $z_0$ be an $a$-point or a $b$-point of $f$ with multiplicity $r(\geq 1)$. Since $f$ and $g$ share $S_1$ IM, that $z_0$ would be a zero of $\psi$ of multiplicity at least $\min \{ m, n\}$. Thus by a simple calculations, we can write
\begin{eqnarray}\label{psudip_lemma1}
\nonumber&&\min \{ m,n \}\cdot \{ \overline{N}\left(r, a; f\right)+\overline{N}\left(r, b; f\right) \}\\
\nonumber& \leq & N\left(r, 0;\psi\right)\\
\nonumber& \leq & T(r, \psi) +O(1) \\
\nonumber& = & N(r, \psi) + m(r, \psi) +O(1) \\
& \leq & \overline{N}(r, f)+ \overline{N}(r, g) +S(r,f)+S(r,g).
\end{eqnarray}
Since $f, g$ share $S_1$ IM and $F, G$ share $1$ CM, then it is obvious to see that
\begin{eqnarray}
\label{psudip1} N(r,\infty ;H) & \leq & \overline{N}\left(r, a; f\right)+ \overline{N}\left(r, b; f\right)+\overline{N}(r,f)+\overline{N}(r,g) \\
\nonumber & &  + N_{0}(r, 0; f')+ N_{0}(r, 0; g')+S(r,f)+S(r,g),
\end{eqnarray}
where $N_{0}(r, 0;f')$ denotes the counting function of the zeros of $f'$ not coming from the zeros of $(f-a)(f-b)$ and $(F-1)$. Similarly $N_{0}(r, 0; g')$ is defined.
Again, \begin{eqnarray}
 \label{pnew}      N(r,1;F|=1)=N(r,1;G|=1) &\leq& N(r,\infty ;H)+S(r,F)+S(r,G),
       \end{eqnarray}
where $ N(r,1;F|=1)$ denotes the counting function of simple $1$-points of $F$. Using the second fundamental theorem for $f$ and $g$, we get
\begin{eqnarray}
\label{psudip2}(m+n+2)T(r,f) &\leq & \overline{N}(r,f)+ \overline{N}\left(r, \frac{1}{F-1}\right)+\overline{N}\left(r, \frac{1}{f-a}\right)\\
\nonumber & & +\overline{N}\left(r, \frac{1}{f-b}\right)-N_{0}(r, 0;f')+S(r,f)
\end{eqnarray}
and
\begin{eqnarray}
\label{psudip3}(m+n+2)T(r,g) &\leq & \overline{N}(r,g)+ \overline{N}\left(r, \frac{1}{G-1}\right)+\overline{N}\left(r, \frac{1}{g-a}\right)\\
\nonumber & & +\overline{N}\left(r, \frac{1}{g-b}\right)-N_{0}(r, 0;g')+S(r,g)
\end{eqnarray}
Using (\ref{psudip2}), (\ref{psudip3}), (\ref{pnew}), (\ref{psudip1}), and (\ref{psudip_lemma1}), we obtain
\begin{eqnarray}
\label{psudip4}& &(m+n+2)\{T(r,f)+T(r,g)\} \\
\nonumber & \leq & \overline{N}(r,f)+\overline{N}(r,g)+\overline{N}\left(r, \frac{1}{f-a}\right)+\overline{N}\left(r, \frac{1}{f-b}\right)+\overline{N}\left(r, \frac{1}{g-a}\right)\\
\nonumber &&+\overline{N}\left(r, \frac{1}{g-b}\right)+N(r,\infty;H)+\overline{N}\left(r, \frac{1}{F-1}|\geq 2\right)+\overline{N}\left(r, \frac{1}{G-1}\right)\\
\nonumber & & -N_{0}(r, 0;f')-N_{0}(r, 0;g') + S(r,f) +S(r,g).\\
\nonumber & \leq & 2\overline{N}(r,f)+2\overline{N}(r,g)+2\overline{N}\left(r, \frac{1}{f-a}\right)+2\overline{N}\left(r, \frac{1}{f-b}\right)+\overline{N}\left(r, \frac{1}{g-a}\right)\\
\nonumber &&+\overline{N}\left(r, \frac{1}{g-b}\right)+\overline{N}\left(r, \frac{1}{F-1}|\geq 2\right)+\overline{N}\left(r, \frac{1}{G-1}\right)\\
\nonumber & &  +S(r,f) +S(r,g).\\
\nonumber & \leq & 3\overline{N}(r,f)+3\overline{N}(r,g)+\overline{N}\left(r, \frac{1}{g-a}\right)+\overline{N}\left(r, \frac{1}{g-b}\right)\\
\nonumber &&+\overline{N}\left(r, \frac{1}{F-1}|\geq 2\right)+\overline{N}\left(r, \frac{1}{G-1}\right) + S(r,f) +S(r,g).\\
\nonumber & \leq & \frac{7}{2}\overline{N}(r,f)+\frac{7}{2}\overline{N}(r,g)+\overline{N}\left(r, \frac{1}{F-1}|\geq 2\right)+\overline{N}\left(r, \frac{1}{G-1}\right)+ S(r,f) +S(r,g).
\end{eqnarray}
Now
\begin{eqnarray}
\label{psudip5}& &\frac{1}{2}\overline{N}\left(r, \frac{1}{F-1}|\leq 1\right)+\overline{N}\left(r, \frac{1}{F-1}|\geq 2\right)\\
\nonumber &\leq & \frac{1}{2}N\left(r, \frac{1}{F-1}\right) \leq \frac{(m+n+1)}{2}T(r,f)+S(r,f)
\end{eqnarray}
and
\begin{eqnarray}
\label{psudip6}& &\frac{1}{2}\overline{N}\left(r, \frac{1}{G-1}|\leq 1\right)+\overline{N}\left(r, \frac{1}{G-1}|\geq 2\right)\\
\nonumber &\leq & \frac{1}{2}N\left(r, \frac{1}{G-1}\right) \leq \frac{(m+n+1)}{2}T(r,g)+S(r,g)
\end{eqnarray}
Using (\ref{psudip5}) and (\ref{psudip6}), we get from (\ref{psudip4})
\begin{eqnarray}
& &(m+n+2)\{T(r,f)+T(r,g)\} \\
\nonumber & \leq & \frac{7}{2}(\overline{N}(r,f)+\overline{N}(r,g))+\frac{(m+n+1)}{2} \{T(r,f)+T(r,g)\}+S(r,f)+S(r,g),
\end{eqnarray}
 which is a contradiction as $m+n\geq 5$. Therefore, our assumption that $H \not \equiv 0$ is not right. \\
\medbreak
\textbf{Case-II} Next we assume that $H\equiv 0$. Then  we have
\bea\label{eq10mm} \frac{1}{F-1}&\equiv&\frac{c_0}{G-1}+c_{1},\eea
where $c_{0}$ and $c_{1}$ are constants, with $c_{0}\not=0$. Thus by Lemma \ref{l2.3}, we have $$T(r,f)=T(r,g)+O(1).$$ Moreover, using given hypothesis, it is clear that $F$ and $G$ share $1$ CM. Since $m+n+1\geq 6$ and $\min\{m,n\}\geq 2$, thus in view of Lemma \ref{lem1}, we get from \ref{eq10mm}
\beas \frac{1}{F-1}&\equiv&\frac{c_0}{G-1},\eeas
i.e.,
\bea\label{eq1022} P(g)&\equiv& c_0 P(f),\eea
where $c_{0}$ is constant, with $c_{0}\not=0$. Since either $n\geq 3$ and $m\geq 2$, or, $m\geq 3$ and $n\geq 2$, applying Lemma \ref{lem3}, we have $c_{0}=1$, i.e., $P(f)\equiv P(g).$  Applying Lemma \ref{lem2}, we get $$f\equiv g.$$ This completes the proof.
\end{proof}
\section{Functions sharing one set and $\{\infty\}$}
In continuation of studying uniqueness of two entire functions when they share three sets, Gross (\cite{Gr}) proposed the following question:
\begin{ques} \label{q2.1}
  Can one find two finite sets $S_1$ and $S_2$ such that any two nonconstant entire functions $f$ and $g$ satisfying $E_{f}(S_{j})=E_{g}(S_{j})$ for
$j = 1, 2$ must be identical ?
\end{ques}
Inspired by Gross' question, many mathematician give explicitly a set $S$ with $n$ elements and make $n$ as small as possible such
that any two meromorphic functions $f$ and $g$ that share the value $\infty$ and the set $S$ must be equal (See \cite{FG}, \cite{FL}, \cite{lin},  \cite{Yi}, \cite{YLu}). In this direction, in 1995, Yi (\cite{Yi}) proved the following results:
\begin{theo} (\cite{Yi}) Let $S=\{z~:~z^n+az^{n-m}+b=0\}$, where $n$ and $m$ are two positive integers such that $m\geq 2$ and $n \geq2m+7$, with $n$ and $m$ having no common factor, $a$ and $b$ be two nonzero constants such that $z^n+az^{n-m}+b=0$ has no multiple roots. If $f$ and $g$ are two nonconstant meromorphic functions satisfying $E_{f}(S)=E_{g}(S)$ and $E_{f}(\{\infty\})=E_{g}(\{\infty\})$, then $f\equiv g$.
\end{theo}
In the same paper Yi (\cite{Yi}) proved that if we take $m=1$, then uniqueness may not occur.
\begin{theo} (\cite{Yi}) Let $S=\{z~:~z^n+az^{n-1}+b=0\}$, where $n\geq9$ be an integer and  $a$ and $b$ be two nonzero constants such that $z^n+az^{n-1}+b=0$ has no multiple roots. If $f$ and $g$ are two nonconstant meromorphic functions satisfying $E_{f}(S)=E_{g}(S)$ and $E_{f}(\{\infty\})=E_{g}(\{\infty\})$, then either $f\equiv g$, or $f\equiv \frac{-ah(h^{n-1}-1)}{h^n-1}$ and $g\equiv \frac{-a(h^{n-1}-1)}{h^n-1}$, where $h$ is a nonconstant meromorphic function.
\end{theo}
But Lahiri(\cite{La}) established a uniqueness result under  some condition  when $m=1$ in Yi's result.
\begin{theo}\label{susm} (\cite{Yi}) Let $S=\{z~:~z^n+az^{n-1}+b=0\}$, where $n\geq8$ be an integer and  $a$ and $b$ be two nonzero constants such that $z^n+az^{n-1}+b=0$ has no multiple roots. If $f$ and $g$ are two nonconstant meromorphic functions having no simple poles such that $E_{f}(S)=E_{g}(S)$ and $E_{f}(\{\infty\})=E_{g}(\{\infty\})$, then  $f\equiv g$.
\end{theo}
Later Fang and Lahiri (\cite{FL}) improved Theorem \ref{susm} by further reducing the cardinality of the same range set and proved the following theorem.
\begin{theo}\label{sus} (\cite{FL}) Let $S=\{z~:~z^n+az^{n-1}+b=0\}$, where $n\geq7$ be an integer and  $a$ and $b$ be two nonzero constants such that $z^n+az^{n-1}+b=0$ has no multiple roots. If $f$ and $g$ are two nonconstant meromorphic functions having no simple poles such that $E_{f}(S)=E_{g}(S)$ and $E_{f}(\{\infty\})=E_{g}(\{\infty\})$, then  $f\equiv g$.
\end{theo}
Later in this direction, many researchers gave many results by relaxing the sharing notations. In this section,  we consider the set  $S_2=\{z \in \mathbb{C}: P(z)=0\}$, where $P(z)$ is defined in (\ref{eq1n}) and establish some uniqueness results when  two nonconstant meromorphic functions sharing the sets $S_{2}$ and $\{\infty\}$.  The consideration of the set $S_2$ is new in connection to how two
meromorphic functions are uniquely determined by their two shared sets $S_2$ and $\{\infty\}$.  The following two theorems are the main
results of this section:
\begin{theo}\label{Th2.1}
Let $S_2=\{z \in \mathbb{C}: P(z)=0\}$, where $P(z)$ is defined in (\ref{eq1n}). If two nonconstant meromorphic functions $f$ and $g$  having no simple poles such that $f$ and $g$ share $\{\infty\}$ IM and $S_2$ CM, and,  $m+n\geq 7$,  $\min\{m,n\}\geq 2$, $\max\{m,n\}\geq 3$, then $f\equiv g$.
\end{theo}
\begin{theo}\label{Th2.2}
Let $S_2=\{z \in \mathbb{C}: P(z)=0\}$, where $P(z)$ is defined in (\ref{eq1n}). If two nonconstant meromorphic functions $f$ and $g$  share  $\{\infty\}$ CM and $S_2$ CM, and,  $m+n\geq 8$,  $\min\{m,n\}\geq 2$, $\max\{m,n\}\geq 3$, then $f\equiv g$.
\end{theo}
\begin{proof} [\textbf{Proof of the Theorem \ref{Th2.1}}]
Let us define the following two functions:
 $$F=-\frac{Q(f)}{c} ~~ \text{and} ~~ G=-\frac{Q(g)}{c},$$
 and
 $$H=\left(\frac{F''}{F'}-\frac{2F'}{F-1}\right)-\left(\frac{G''}{G'}-\frac{2G'}{G-1}\right).$$
Since $f, g$ share $S_2$ CM, then $F, G$ share $1$ CM.\\
\medbreak
\textbf{Case-I} 
If possible, let $H   \not \equiv 0$. Then obviously $F-1$ and $G-1$ are not linearly dependent. Since $f, g$ share $\infty$ IM and $F, G$ share $1$ CM, then it is obvious to see that
\begin{eqnarray}
\label{sudip1pp} N(r,\infty ;H) & \leq & \overline{N}\left(r, a; f\right)+ \overline{N}\left(r, b; f\right)+\overline{N}\left(r, a; g\right)+ \overline{N}\left(r, b; g\right)+\overline{N}(r,f) \\
\nonumber & &  + N_{0}(r, 0; f')+ N_{0}(r, 0; g')+S(r,f)+S(r,g),
\end{eqnarray}
where $N_{0}(r, 0;f')$ denotes the counting function of the zeros of $f'$ not coming from the zeros of $(f-a)(f-b)$ and $(F-1)$. Similarly $N_{0}(r, 0; g')$ is defined.
Again, \begin{eqnarray}
 \label{newpp}      N(r,1;F|=1)=N(r,1;G|=1) &\leq& N(r,\infty ;H)+S(r,F)+S(r,G),
       \end{eqnarray}
where $ N(r,1;F|=1)$ denotes the counting function of simple $1$-points of $F$. Using the second fundamental theorem for $f$ and $g$ we get
\begin{eqnarray}
\label{sudip2pp}(m+n+2)T(r,f) &\leq & \overline{N}(r,f)+ \overline{N}\left(r, \frac{1}{F-1}\right)+\overline{N}\left(r, \frac{1}{f-a}\right)\\
\nonumber & & +\overline{N}\left(r, \frac{1}{f-b}\right)-N_{0}(r, 0;f')+S(r,f)
\end{eqnarray}
and
\begin{eqnarray}
\label{sudip3pp}(m+n+2)T(r,g) &\leq & \overline{N}(r,g)+ \overline{N}\left(r, \frac{1}{G-1}\right)+\overline{N}\left(r, \frac{1}{g-a}\right)\\
\nonumber & & +\overline{N}\left(r, \frac{1}{g-b}\right)-N_{0}(r, 0;g')+S(r,g)
\end{eqnarray}
Using (\ref{sudip2pp}), (\ref{sudip3pp}), (\ref{newpp}), and (\ref{sudip1pp}), we obtain
\begin{eqnarray}
\label{sudip4pp}& &(m+n+2)\{T(r,f)+T(r,g)\} \\
\nonumber & \leq & \overline{N}(r,f)+\overline{N}(r,g)+\overline{N}\left(r, \frac{1}{f-a}\right)+\overline{N}\left(r, \frac{1}{f-b}\right)+\overline{N}\left(r, \frac{1}{g-a}\right)\\
\nonumber &&+\overline{N}\left(r, \frac{1}{g-b}\right)+N(r,\infty;H)+\overline{N}\left(r, \frac{1}{F-1}|\geq 2\right)+\overline{N}\left(r, \frac{1}{G-1}\right)\\
\nonumber & & -N_{0}(r, 0;f')-N_{0}(r, 0;g') + S(r,f) +S(r,g).\\
\nonumber & \leq & 2\overline{N}(r,f)+\overline{N}(r,g)+2\overline{N}\left(r, \frac{1}{f-a}\right)+2\overline{N}\left(r, \frac{1}{f-b}\right)+2\overline{N}\left(r, \frac{1}{g-a}\right)\\
\nonumber &&+2\overline{N}\left(r, \frac{1}{g-b}\right)+\overline{N}\left(r, \frac{1}{F-1}|\geq 2\right)+\overline{N}\left(r, \frac{1}{G-1}\right)\\
\nonumber & &  +S(r,f) +S(r,g).\\
\nonumber & \leq & \frac{3}{2}\overline{N}(r,f)+\frac{3}{2}\overline{N}(r,g)+4(T\left(r, f\right)+T\left(r, g\right))+\overline{N}\left(r, \frac{1}{F-1}|\geq 2\right)\\
\nonumber && +\overline{N}\left(r, \frac{1}{G-1}\right)+ S(r,f) +S(r,g).
\end{eqnarray}
Now
\begin{eqnarray}
\label{sudip5pp}& &\frac{1}{2}\overline{N}\left(r, \frac{1}{F-1}|\leq 1\right)+\overline{N}\left(r, \frac{1}{F-1}|\geq 2\right)\\
\nonumber &\leq & \frac{1}{2}N\left(r, \frac{1}{F-1}\right) \leq \frac{(m+n+1)}{2}T(r,f)+S(r,f)
\end{eqnarray}
and
\begin{eqnarray}
\label{sudip6pp}& &\frac{1}{2}\overline{N}\left(r, \frac{1}{G-1}|\leq 1\right)+\overline{N}\left(r, \frac{1}{G-1}|\geq 2\right)\\
\nonumber &\leq & \frac{1}{2}N\left(r, \frac{1}{G-1}\right) \leq \frac{(m+n+1)}{2}T(r,g)+S(r,g)
\end{eqnarray}
Using (\ref{sudip5pp}) and (\ref{sudip6pp}), we get from (\ref{sudip4pp})
\begin{eqnarray}
& &(m+n-2)\{T(r,f)+T(r,g)\} \\
\nonumber & \leq & \frac{3}{2}(\overline{N}(r,f)+\overline{N}(r,g))+\frac{(m+n+1)}{2} \{T(r,f)+T(r,g)\}+S(r,f)+S(r,g),
\end{eqnarray}
Since $f$ and $g$ has no simple poles, thus
\begin{eqnarray}
& &(\frac{m+n}{2}-\frac{5}{2})\{T(r,f)+T(r,g)\} \\
\nonumber & \leq & \frac{3}{2}(\overline{N}(r,\infty;f\mid\geq2)+\overline{N}(r,\infty;g\mid\geq2))+S(r,f)+S(r,g)\\
\nonumber & \leq & \frac{3}{4}(N(r,\infty;f)+N(r,\infty;g))+S(r,f)+S(r,g),
\end{eqnarray}
 which is a contradiction as $m+n\geq 7$. Therefore, our assumption that $H \not \equiv 0$ is not right. \\
\medbreak
\textbf{Case-II} Next we assume that $H\equiv 0$. Then  we have
\bea\label{eq10mpp} \frac{1}{F-1}&\equiv&\frac{c_0}{G-1}+c_{1},\eea
where $c_{0}$ and $c_{1}$ are constants, with $c_{0}\not=0$. Thus by Lemma \ref{l2.3}, we have $$T(r,f)=T(r,g)+O(1).$$ Moreover, using given hypothesis, it is clear that $F$ and $G$ share $1$ CM. Since $m+n+1\geq 6$ and $\min\{m,n\}\geq 2$, thus in view of Lemma \ref{lem1}, we get from \ref{eq10mpp}
\beas \frac{1}{F-1}&\equiv&\frac{c_0}{G-1},\eeas
i.e.,
\bea\label{eq1022pp} P(g)&\equiv& c_0 P(f),\eea
where $c_{0}$ is constant, with $c_{0}\not=0$. Since either $n\geq 3$ and $m\geq 2$, or, $m\geq 3$ and $n\geq 2$, applying Lemma \ref{lem3}, we have $c_{0}=1$, i.e., $P(f)\equiv P(g).$  Applying Lemma \ref{lem2}, we get $$f\equiv g.$$ This completes the proof.
\end{proof}
\begin{proof} [\textbf{Proof of the Theorem \ref{Th2.2}}]
Let us define the following functions:
 $$F=-\frac{Q(f)}{c} ~~ \text{and} ~~ G=-\frac{Q(g)}{c},$$
 and
 $$H=\left(\frac{F''}{F'}-\frac{2F'}{F-1}\right)-\left(\frac{G''}{G'}-\frac{2G'}{G-1}\right).$$
Since $f, g$ share $S_2$ CM, then $F, G$ share $1$ CM.\\
\medbreak
\textbf{Case-I} 
If possible, let $H   \not \equiv 0$. Then obviously $F-1$ and $G-1$ are not linearly dependent. Since $f, g$ share $\infty$ CM and $F, G$ share $1$ CM, then it is obvious to see that
\begin{eqnarray}
\label{sudip1} N(r,\infty ;H) & \leq & \overline{N}\left(r, a; f\right)+ \overline{N}\left(r, b; f\right)+\overline{N}\left(r, a; g\right)+ \overline{N}\left(r, b; g\right) \\
\nonumber & &  + N_{0}(r, 0; f')+ N_{0}(r, 0; g')+S(r,f)+S(r,g),
\end{eqnarray}
where $N_{0}(r, 0;f')$ denotes the counting function of the zeros of $f'$ not coming from the zeros of $(f-a)(f-b)$ and $(F-1)$. Similarly $N_{0}(r, 0; g')$ is defined.
Again, \begin{eqnarray}
 \label{new}      N(r,1;F|=1)=N(r,1;G|=1) &\leq& N(r,\infty ;H)+S(r,F)+S(r,G),
       \end{eqnarray}
where $ N(r,1;F|=1)$ denotes the counting function of simple $1$-points of $F$. Using second fundamental theorem for $f$ and $g$ we get
\begin{eqnarray}
\label{sudip2}(m+n+2)T(r,f) &\leq & \overline{N}(r,f)+ \overline{N}\left(r, \frac{1}{F-1}\right)+\overline{N}\left(r, \frac{1}{f-a}\right)\\
\nonumber & & +\overline{N}\left(r, \frac{1}{f-b}\right)-N_{0}(r, 0;f')+S(r,f)
\end{eqnarray}
and
\begin{eqnarray}
\label{sudip3}(m+n+2)T(r,g) &\leq & \overline{N}(r,g)+ \overline{N}\left(r, \frac{1}{G-1}\right)+\overline{N}\left(r, \frac{1}{g-a}\right)\\
\nonumber & & +\overline{N}\left(r, \frac{1}{g-b}\right)-N_{0}(r, 0;g')+S(r,g)
\end{eqnarray}
Using (\ref{sudip2}), (\ref{sudip3}), (\ref{new}), and (\ref{sudip1}), we obtain
\begin{eqnarray}
\label{sudip4}& &(m+n+2)\{T(r,f)+T(r,g)\} \\
\nonumber & \leq & \overline{N}(r,f)+\overline{N}(r,g)+\overline{N}\left(r, \frac{1}{f-a}\right)+\overline{N}\left(r, \frac{1}{f-b}\right)+\overline{N}\left(r, \frac{1}{g-a}\right)\\
\nonumber &&+\overline{N}\left(r, \frac{1}{g-b}\right)+N(r,\infty;H)+\overline{N}\left(r, \frac{1}{F-1}|\geq 2\right)+\overline{N}\left(r, \frac{1}{G-1}\right)\\
\nonumber & & -N_{0}(r, 0;f')-N_{0}(r, 0;g') + S(r,f) +S(r,g).\\
\nonumber & \leq & \overline{N}(r,f)+\overline{N}(r,g)+2\overline{N}\left(r, \frac{1}{f-a}\right)+2\overline{N}\left(r, \frac{1}{f-b}\right)+2\overline{N}\left(r, \frac{1}{g-a}\right)\\
\nonumber &&+2\overline{N}\left(r, \frac{1}{g-b}\right)+\overline{N}\left(r, \frac{1}{F-1}|\geq 2\right)+\overline{N}\left(r, \frac{1}{G-1}\right)\\
\nonumber & &  +S(r,f) +S(r,g).\\
\nonumber & \leq & \overline{N}(r,f)+\overline{N}(r,g)+4(T\left(r, f\right)+T\left(r, g\right))\\
\nonumber &&+\overline{N}\left(r, \frac{1}{F-1}|\geq 2\right)+\overline{N}\left(r, \frac{1}{G-1}\right) + S(r,f) +S(r,g).
\end{eqnarray}
Now
\begin{eqnarray}
\label{sudip5}& &\frac{1}{2}\overline{N}\left(r, \frac{1}{F-1}|\leq 1\right)+\overline{N}\left(r, \frac{1}{F-1}|\geq 2\right)\\
\nonumber &\leq & \frac{1}{2}N\left(r, \frac{1}{F-1}\right) \leq \frac{(m+n+1)}{2}T(r,f)+S(r,f)
\end{eqnarray}
and
\begin{eqnarray}
\label{sudip6}& &\frac{1}{2}\overline{N}\left(r, \frac{1}{G-1}|\leq 1\right)+\overline{N}\left(r, \frac{1}{G-1}|\geq 2\right)\\
\nonumber &\leq & \frac{1}{2}N\left(r, \frac{1}{G-1}\right) \leq \frac{(m+n+1)}{2}T(r,g)+S(r,g)
\end{eqnarray}
Thus using (\ref{sudip5}) and (\ref{sudip6}), we get from (\ref{sudip4}) that
\begin{eqnarray}
& &(m+n-2)\{T(r,f)+T(r,g)\} \\
\nonumber & \leq & (\overline{N}(r,\infty;f)+\overline{N}(r,\infty;g))+\frac{(m+n+1)}{2} \{T(r,f)+T(r,g)\}+S(r,f)+S(r,g),
\end{eqnarray}
 which is a contradiction as $m+n\geq 8$. Therefore, our assumption that $H \not \equiv 0$ is not right. \\
\medbreak
\textbf{Case-II} Next we assume that $H\equiv 0$. Then  we have
\bea\label{eq10m} \frac{1}{F-1}&\equiv&\frac{c_0}{G-1}+c_{1},\eea
where $c_{0}$ and $c_{1}$ are constants, with $c_{0}\not=0$. Thus by Lemma \ref{l2.3}, we have $$T(r,f)=T(r,g)+O(1).$$ Moreover, using given hypothesis, it is clear that $F$ and $G$ share $1$ CM. Since $m+n+1\geq 6$ and $\min\{m,n\}\geq 2$, thus in view of Lemma \ref{lem1}, we get from \ref{eq10m}
\beas \frac{1}{F-1}&\equiv&\frac{c_0}{G-1},\eeas
i.e.,
\bea\label{eq1022} P(g)&\equiv& c_0 P(f),\eea
where $c_{0}$ is constant, with $c_{0}\not=0$. Since either $n\geq 3$ and $m\geq 2$, or, $m\geq 3$ and $n\geq 2$, applying Lemma \ref{lem3}, we have $c_{0}=1$, i.e., $P(f)\equiv P(g).$  Applying Lemma \ref{lem2}, we get $$f\equiv g.$$ This completes the proof.
\end{proof}
\section{Acknowledgements}
The authors are grateful to the anonymous referees for their valuable suggestions which considerably improved the presentation of the paper.\par
 Mr. Amit Kumar Pal is thankful to the Higher Education Dept., Govt. of West Bengal, for granting the SVMCM research fellowship(No.: WBP211653035749), during the tenure of which the work was done.\\
Mr. Sudip Saha is thankful to the Council of Scientific and Industrial Research, HRDG, India for granting Senior Research
Fellowship (File No.: 08/525(0003)/2019-EMR-I) during the tenure of which this work was done.\par
The research work of  Dr. B. Chakraborty is  supported by the Department of Higher Education, Science and Technology \text{\&} Biotechnology, Govt. of West Bengal under the sanction order no. 1303(sanc.)/STBT-11012(26)/17/2021-ST SEC dated 14/03/2022.


\begin{thebibliography}{99}
\bibitem{1} A. Banerjee, A new class of strong uniqueness polynomials satisfying Fujimoto's condition, Annales Academi{\ae} Scientiarum Fennic{\ae} Mathematica, 40(2015), 465-474.
\bibitem{1.1} A. Banerjee and B. Chakraborty, A new type of unique range set with deficient values, Afr. Mat., 26 (2015), no. 7-8, 1561-1572.
\bibitem{1.2} A. Banerjee and B. Chakraborty, On some sufficient conditions of strong uniqueness polynomials, Adv. Pure Appl. Math. 8 (2017), no. 1, 1-13.
\bibitem{chak} B. Chakraborty, J. Kamila, A. K. Pal, and S. Saha, Some results on the unique range sets, J. Korean Math. Soc. 58 (2021), no. 3, 741-760.
\bibitem{FG} M. Fang and H. Guo, On meromorphic functions sharing two values, Analysis 17 (1997), 355–366.
\bibitem{La} I. Lahiri, The range set of meromorphic derivatives, Northeast. Math. J. 14(3) (1998), 353–360.
\bibitem{FL} M. Fang and I. Lahiri, Unique range set for certain meromorphic functions, Indian J. Math., 45(2) (2003), 141–150.
\bibitem{2} G. Frank and M. Reinders, A unique range set for meromorphic functions with $11$ elements, Complex Var. and Ellip. Eqn., 37(1)(1998), 185-193.
\bibitem{2.1} H. Fujimoto, On uniqueness of meromorphic functions sharing finite sets, Amer. J. Math, 122(2000), 1175-1203.
\bibitem{2.2} H. Fujimoto, On uniqueness polynomials for meromorphic functions, Nagoya Math. J., 170 (2003), 33-46.
\bibitem{4.01} F.Gross, On the Distribution of Values of Meromorphic Functions, Trans. Amer. Math. Soc., 131 (1968), 199-214.
\bibitem{Gr} F. Gross, Factorization of Meromorphic Functions and Some Open Problems, Proc. Conf. Univ. Kentucky, Leixngton, 1976; Lect. Notes Math. 599 (1977), 51–69.
\bibitem{GY} F. Gross and C. C. Yang, \emph{On preimage and range sets of meromorphic functions}, Proc. Japan Acad. \textbf{58} (1982), 17--20.
\bibitem{4.2} A. Z. Mokhon'ko, On the Nevanlinna characteristics of some meromorphic functions, in "Theory of func-
tions, functional analysis and their applications", Izd-vo Khar'kovsk, Un-ta, 14(1971), 83-87.
\bibitem{lin} W. C. Lin and H. X. Yi, Some further results on meromorphic functions that share two sets, Kyungpook Math. J., 43 (2003), 73-85.
\bibitem{YY} C. C. Yang and H. X. Yi, \emph{Uniqueness Theory of Meromorphic Functions, Mathematics and its Applications}, \textbf{557}, Kluwer Academic Publishers Group, Dordrecht, 2003.
\bibitem{Yi} H. X. Yi, Unicity theorems for meromorphic or entire functions II, Bull. Austral. Math. Soc., 52(1995), 215–224.
\bibitem{YL} H. X. Yi and W. C. Lin, Uniqueness of meromorphic functions and a question of Gross, Kyungpook Math. J. 46 (2006), 437–444.
\bibitem{YLu} H. X. Yi and W. R. Lu, Meromorphic functions that share two sets II, Acta Math. Sci. Ser. B (Engl. Ed.) 24(1) (2004), 83–90.
\bibitem{2YL} B. Yi and Y. H. Li, \emph{The uniqueness of meromorphic functions that share two sets with CM}, Acta Math. Sin., Chin. Ser. \textbf{55} (2012), 363--368.
\end{thebibliography}
\end{document}